\documentclass[12pt]{article}
\usepackage{amsmath,amsfonts,amssymb,amsthm,txfonts,hyperref,url,t1enc}
\newtheorem{corollary}{Corollary}
\newtheorem{lemma}{Lemma}
\newtheorem{theorem}{Theorem}
\newcommand{\N}{\mathbb N}
\newcommand{\Z}{\mathbb Z}
\newcommand{\Q}{\mathbb Q}
\newcommand{\R}{\mathbb R}
\newcommand{\C}{\mathbb C}
\newcommand{\K}{\textbf{\textit{K}}}
\begin{document}
\par
\noindent
\centerline{{\large A hypothetical upper bound for the solutions of a Diophantine}}
\vskip 0.3truecm
\noindent
\centerline{{\large equation with a finite number of solutions}}
\vskip 0.3truecm
\noindent
\centerline{{\large Apoloniusz Tyszka}}
\vskip 0.3truecm
\noindent
{\bf Abstract.} We conjecture that if a system $S \subseteq \{x_i=1,~x_i+x_j=x_k,~x_i \cdot x_j=x_k:$ $i,j,k \in \{1,\ldots,n\}\}$
has only finitely many solutions in integers \mbox{$x_1,\ldots,x_n$},
then each such solution \mbox{$(x_1,\ldots,x_n)$} satisfies $|x_1|,\ldots,|x_n| \leq 2^{\textstyle 2^{n-1}}$.
By the conjecture, if a Diophantine equation has only finitely many solutions in integers (non-negative integers, rationals),
then their heights are bounded from above by a computable function of the degree and the coefficients of the equation.
The conjecture implies that the set of Diophantine equations which have infinitely many solutions in integers
(non-negative integers) is recursively enumerable. The conjecture formulated for an arbitrary computable bound instead
of $2^{\textstyle 2^{n-1}}$ remains in contradiction to Matiyasevich's conjecture that each recursively enumerable
set \mbox{${\cal M} \subseteq {\N}^n$} has a finite-fold Diophantine representation.
\vskip 0.3truecm
\begin{sloppypar}
\end{sloppypar}
\vskip 0.9truecm
\par
The Davis-Putnam-Robinson-Matiyasevich theorem states that every recursively enumerable
set \mbox{${\cal M} \subseteq {\N}^n$} has a Diophantine representation, that is
\begin{equation}
\tag*{\tt (R)}
(a_1,\ldots,a_n) \in {\cal M} \Longleftrightarrow
\exists x_1, \ldots, x_m \in \N ~~W(a_1,\ldots,a_n,x_1,\ldots,x_m)=0
\end{equation}
\begin{sloppypar}
\noindent
for some polynomial $W$ with integer coefficients, see \cite{Matiyasevich1} and \cite{Kuijer}.
The polynomial~$W$ can be computed, if we know a Turing machine~$M$
such that, for all \mbox{$(a_1,\ldots,a_n) \in {\N}^n$}, $M$ halts on \mbox{$(a_1,\ldots,a_n)$} if and only if
\mbox{$(a_1,\ldots,a_n) \in {\cal M}$}, see \cite{Matiyasevich1} and \cite{Kuijer}. The representation~{\tt (R)}
is said to be \mbox{finite-fold} if for any \mbox{$a_1,\ldots,a_n \in \N$} the equation
\mbox{$W(a_1,\ldots,a_n,x_1,\ldots,x_m)=0$} has at most finitely many solutions \mbox{$(x_1,\ldots,x_m) \in {\N}^m$}.
Yu.~Matiyasevich conjectures that each recursively enumerable set \mbox{${\cal M} \subseteq {\N}^n$}
has a \mbox{finite-fold} Diophantine representation, see \mbox{\cite[pp.~341--342]{Davis2}},
\mbox{\cite[p.~42]{Matiyasevich2}} and \mbox{\cite[p.~79]{Matiyasevich3}}. His conjecture remains
in contradiction to the following Conjecture, see Corollary~\ref{cor2}.
\end{sloppypar}
\vskip 0.2truecm
\par
\noindent
{\bf Conjecture.} {\em For each positive integer $n$, if a system
\[
S \subseteq E_n=\{x_i=1,~x_i+x_j=x_k,~x_i \cdot x_j=x_k:~i,j,k \in \{1,\ldots,n\}\}
\]
has only finitely many solutions in integers \mbox{$x_1,\ldots,x_n$},
then each such solution \mbox{$(x_1,\ldots,x_n)$} satisfies $|x_1|,\ldots,|x_n| \leq 2^{\textstyle 2^{n-1}}$.}
\vskip 0.2truecm
\par
Let $T_n$ denote the set of all integer tuples \mbox{$(a_1,\ldots,a_n)$} for which
there exists a system \mbox{$S \subseteq E_n$} such that \mbox{$(a_1,\ldots,a_n)$} solves $S$
and $S$ has at most finitely many solutions in integers \mbox{$x_1,\ldots,x_n$}.
If \mbox{$(a_1,\ldots,a_n) \in T_n$}, then \mbox{$(a_1,\ldots,a_n)$} solves the system
\begin{displaymath}
\left\{
\begin{array}{rcl}
x_{i} &=& 1~~~~~~({\rm all~} i \in \{1,\ldots,n\} {\rm ~with~} a_i=1) \\
x_i+x_j &=& x_k~~~~~({\rm all~} i,j,k \in \{1,\ldots,n\} {\rm ~with~} a_i+a_j=a_k) \\
x_i \cdot x_j &=& x_k~~~~~({\rm all~} i,j,k \in \{1,\ldots,n\} {\rm ~with~} a_i \cdot a_j=a_k)
\end{array}
\right.
\end{displaymath}
which has only finitely many solutions in integers \mbox{$x_1,\ldots,x_n$}.
\vskip 0.2truecm
\noindent
{\bf Proposition.} {\em The Conjecture is true for \mbox{$n \leq 3$}.}
\begin{proof}
$T_1=\{0,1\}$. $T_2$ consists of the pairs \mbox{$(0,0)$}, \mbox{$(1,1)$}, \mbox{$(-1,1)$},
\mbox{$(0,1)$}, \mbox{$(1,2)$}, \mbox{$(2,4)$} and their permutations. $T_3$ consists of the triples
\vskip 0.2truecm
\centerline{$(0,0,0)$, $(1,1,1)$,}
\vskip 0.2truecm
\centerline{$(-1,-1,1)$, $(0,0,1)$, $(1,1,-1)$, $(1,1,0)$, $(1,1,2)$, $(2,2,1)$, $(2,2,4)$, $(4,4,2)$,}
\vskip 0.2truecm
\centerline{$(1,-2,-1)$, $(1,-1,0)$, $(1,-1,2)$, $(1,0,2)$, $(1,2,3)$, $(1,2,4)$,}
\vskip 0.2truecm
\centerline{$(2,4,-2)$, $(2,4,0)$, $(2,4,6)$, $(2,4,8)$, $(2,4,16)$,}
\vskip 0.2truecm
\centerline{$(-4,-2,2)$, $(-2,-1,2)$, $(3,6,9)$, $(4,8,16)$}
\vskip 0.2truecm
\noindent
and their permutations.
\end{proof}
\par
Let ${\cal C}_n$ denote the Conjecture restricted to the systems $S \subseteq E_n$.
\begin{lemma}\label{lem1}
For each positive integer $n$, if ${\cal C}_n$ is false then ${\cal C}_{n+1}$ is false.
\end{lemma}
\begin{proof}
Since ${\cal C}_n$ is false, there exist integers \mbox{$a_1,\ldots,a_n$} and a system \mbox{$S \subseteq E_n$}
such that \mbox{$(a_1,\ldots,a_n)$} solves $S$, \mbox{$|a_n|>2^{\textstyle 2^{n-1}}$}, and the system $S$
has only finitely many solutions in integers \mbox{$x_1,\ldots,x_n$}.
Then, \mbox{$|a_n^2|>2^{\textstyle 2^n}$} and the integer \mbox{$(n+1)$-tuple} \mbox{$(a_1,\ldots,a_n,a_n^2)$}
solves the system \mbox{$S \cup \{x_n \cdot x_n=x_{n+1}\}$} which has only finitely many solutions
in integers \mbox{$x_1,\ldots,x_n,x_{n+1}$}.
\end{proof}
\begin{sloppypar}
The Conjecture seems to be true for $\N$, $\N \setminus \{0\}$, $\Q$, $\R$ and~$\C$,
cf. \mbox{\cite[p.~528, \mbox{Conjecture 5d}]{Tyszka3}}, \mbox{\cite[p.~177,~\mbox{Conjecture 1.5(d)}]{Tyszka4}}
and \mbox{\cite[p.~180]{Tyszka4}}.
For $n \geq 2$, the bound $2^{\textstyle 2^{n-1}}$ cannot be decreased because the system
\end{sloppypar}
\begin{displaymath}
\left\{
\begin{array}{rcl}
x_1+x_1 &=& x_2 \\
x_1 \cdot x_1 &=& x_2 \\
x_2 \cdot x_2 &=& x_3 \\
x_3 \cdot x_3 &=& x_4 \\
&\ldots& \\
x_{n-1} \cdot x_{n-1} &=&x_n
\end{array}
\right.
\end{displaymath}
\begin{sloppypar}
\noindent
has precisely two integer solutions, namely $(0,\ldots,0)$ and
$\Bigl(2,4,16,256,\ldots,2^{\textstyle 2^{n-2}},2^{\textstyle 2^{n-1}}\Bigr)$.
Nevertheless, for each integer $n \geq 12$ there is a system $S \subseteq E_n$ which has infinitely many
integer solutions and they all belong to \mbox{${\Z}^n \setminus [-2^{\textstyle 2^{n-1}},2^{\textstyle 2^{n-1}}]^n$},
see \cite[p.~4,~Theorem~1]{Tyszka2}, cf. \cite[p.~178,~Theorem~2.4]{Tyszka4}. The next theorem generalizes this.
\end{sloppypar}
\begin{theorem}\label{the1} (\cite[p.~7,~Theorem~2]{Tyszka2}) There is an algorithm that for every computable
function \mbox{$f:\N \to \N$} returns a positive integer $m(f)$, for which a second algorithm accepts
on the input $f$ and any integer \mbox{$n \geq m(f)$}, and returns a system \mbox{$S \subseteq E_n$} such that
$S$ has infinitely many integer solutions and each integer tuple $(x_1,\ldots,x_n)$ that solves $S$ satisfies $x_1=f(n)$.
\end{theorem}
\begin{theorem}\label{the2} (\cite[p.~8,~Corollary]{Tyszka2}) There is an algorithm that for every computable
function \mbox{$f:\N \to \N$} returns a positive integer $m(f)$, for which a second algorithm accepts
on the input $f$ and any integer \mbox{$n \geq m(f)$}, and returns an integer tuple \mbox{$(x_1,\ldots,x_n)$}
for which $x_1=f(n)$ and
\vskip 0.2truecm
\par
\noindent
~~~~~(1)~~for each integers $y_1,\ldots,y_n$ the conjunction
\[
\Bigl(\forall i \in \{1,\ldots,n\}~(x_i=1 \Longrightarrow y_i=1)\Bigr) ~\wedge
\]
\[
\Bigl(\forall i,j,k \in \{1,\ldots,n\}~(x_i+x_j=x_k \Longrightarrow y_i+y_j=y_k)\Bigr) ~\wedge
\]
\[
\forall i,j,k \in \{1,\ldots,n\}~(x_i \cdot x_j=x_k \Longrightarrow y_i \cdot y_j=y_k)
\]
\par
\noindent
~~~~~~~~~~~~implies that $x_1=y_1$.
\end{theorem}
If $n \geq 2$, then the tuple \mbox{$\left(x_1,\ldots,x_n\right)=\left(2^{\textstyle 2^{n-2}},2^{\textstyle 2^{n-3}},\ldots,256,16,4,2,1\right)$}
has property {\em (1)}. Unfortunately, we do not know any explicitly given integers \mbox{$x_1,\ldots,x_n$}
with property {\em (1)} and \mbox{$|x_1|>2^{\textstyle 2^{n-2}}$}.
\vskip 0.2truecm
\par
To each system $S \subseteq E_n$ we assign the system $\widetilde{S}$ defined by
\vskip 0.2truecm
\par
\noindent
\centerline{$\left(S \setminus \{x_i=1:~i \in \{1,\ldots,n\}\}\right) \cup$}
\par
\noindent
\centerline{$\{x_i \cdot x_j=x_j:~i,j \in \{1,\ldots,n\} {\rm ~and~the~equation~} x_i=1 {\rm ~belongs~to~} S\}$}
\vskip 0.2truecm
\par
\noindent
In other words, in order to obtain $\widetilde{S}$ we remove from $S$ each
equation $x_i=1$ and replace it by the following $n$ equations:
\vskip 0.2truecm
\par
\noindent
\centerline{$\begin{array}{rcl}
x_i \cdot x_1 &=& x_1\\
&\ldots& \\
x_i \cdot x_n &=& x_n
\end{array}$}
\begin{lemma}\label{lem2}
For each system $S \subseteq E_n$
\begin{eqnarray*}
\{(x_1,\ldots,x_n) \in {\Z}^n:~(x_1,\ldots,x_n) {\rm ~solves~} \widetilde{S}\} &=& \\
\{(x_1,\ldots,x_n) \in {\Z}^n:~(x_1,\ldots,x_n) {\rm ~solves~} S\} \cup
\{(0,\ldots,0)\}&
\end{eqnarray*}
\end{lemma}
\par
By Lemma~\ref{lem2}, the Conjecture is equivalent to
\[
\forall x_1,\ldots,x_n \in \Z ~\exists y_1,\ldots,y_n \in \Z
\]
\[
\bigl(2^{\textstyle 2^{n-1}}<|x_1| \Longrightarrow (|x_1|<|y_1| \vee \ldots \vee |x_1|<|y_n|)\bigr) ~\wedge
\]
\[
\bigl(\forall i,j,k \in \{1,\ldots,n\}~(x_i+x_j=x_k \Longrightarrow y_i+y_j=y_k)\bigr) ~\wedge
\]
\[
\forall i,j,k \in \{1,\ldots,n\}~(x_i \cdot x_j=x_k \Longrightarrow y_i \cdot y_j=y_k)
\]
\par
The statement
\[
\forall x_1,\ldots,x_n \in \Z ~\exists y_1,\ldots,y_n \in \Z
\]
\[
\bigl(2^{\textstyle 2^{n-1}}<|x_1| \Longrightarrow |x_1|<|y_1|\bigr) ~\wedge
\]
\[
\bigl(\forall i,j,k \in \{1,\ldots,n\}~(x_i+x_j=x_k \Longrightarrow y_i+y_j=y_k)\bigr) ~\wedge
\]
\[
\forall i,j,k \in \{1,\ldots,n\}~(x_i \cdot x_j=x_k \Longrightarrow y_i \cdot y_j=y_k)
\]
\begin{sloppypar}
\noindent
obviously strengthens the Conjecture, but is false for some~$n$. The last observation follows from Theorem~\ref{the2}.
\vskip 0.2truecm
\par
Let \mbox{$D(x_1,\ldots,x_p) \in {\Z}[x_1,\ldots,x_p]$}. For the Diophantine equation \mbox{$2 \cdot D(x_1,\ldots,x_p)=0$},
let $M$ denote the maximum of the absolute values of its coefficients.
Let ${\cal T}$ denote the family of all polynomials
$W(x_1,\ldots,x_p) \in {\Z}[x_1,\ldots,x_p]$ whose all coefficients belong to the interval $[-M,M]$
and ${\rm deg}(W,x_i) \leq d_i={\rm deg}(D,x_i)$ for each $i \in \{1,\ldots,p\}$.
Here we consider the degrees of $W(x_1,\ldots,x_p)$ and $D(x_1,\ldots,x_p)$
with respect to the variable~$x_i$. It is easy to check that
\begin{equation}
\tag*{$(\ast)$}
{\rm card}({\cal T})=(2M+1)^{\textstyle (d_1+1) \cdot \ldots \cdot (d_p+1)}
\end{equation}
\par
We choose any bijection \mbox{$\tau: \{p+1,\ldots,{\rm card}({\cal T})\} \longrightarrow {\cal T} \setminus \{x_1,\ldots,x_p\}$}.
Let ${\cal H}$ denote the family of all equations of the forms
\vskip 0.2truecm
\noindent
\centerline{$x_i=1$, $x_i+x_j=x_k$, $x_i \cdot x_j=x_k$~~($i,j,k \in \{1,\ldots,{\rm card}({\cal T})\})$}
\vskip 0.2truecm
\noindent
which are polynomial identities in \mbox{${\Z}[x_1,\ldots,x_p]$} if
\[
\forall s \in \{p+1,\ldots,{\rm card}({\cal T})\} ~~x_s=\tau(s)
\]
There is a unique \mbox{$q \in \{p+1,\ldots,{\rm card}({\cal T})\}$} such that \mbox{$\tau(q)=2 \cdot D(x_1,\ldots,x_p)$}.
For each ring $\K$ extending $\Z$ the system ${\cal H}$ implies \mbox{$2 \cdot D(x_1,\ldots,x_p)=x_q$}.
To see this, we observe that there exist pairwise distinct
\mbox{$t_0,\ldots,t_m \in {\cal T}$} such that $m>p$ and
\[
t_0=1~ \wedge ~t_1=x_1~ \wedge ~\ldots~ \wedge ~t_p=x_p~ \wedge ~t_m=2 \cdot D(x_1,\ldots,x_p)~ \wedge
\]
\[
\forall i \in \{p+1,\ldots,m\}~ \exists j,k \in \{0,\ldots,i-1\} ~~(t_j+t_k=t_i \vee t_i+t_k=t_j \vee t_j \cdot t_k=t_i)
\]
For each ring $\K$ extending $\Z$ and for each \mbox{$x_1,\ldots,x_p \in \K$}
there exists a unique tuple \mbox{($x_{p+1},\ldots,x_{{\rm card}({\cal T})}) \in \K^{{\rm card}({\cal T})-p}$}
such that the tuple \mbox{$(x_1,\ldots,x_p,x_{p+1},\ldots,x_{{\rm card}({\cal T})})$}
solves the system \mbox{${\cal H}$}. The sought elements \mbox{$x_{p+1},\ldots,x_{{\rm card}({\cal T})}$}
are given by the formula
\[
\forall s \in \{p+1,\ldots,{\rm card}({\cal T})\} ~~x_s=\tau(s)(x_1,\ldots,x_p)
\]
\begin{lemma}\label{lem3}
The system ${\cal H} \cup \{x_q+x_q=x_q\}$ can be simply computed.
For each ring $\K$ extending $\Z$, the equation $D(x_1,\ldots,x_p)=0$
is equivalent to the system ${\cal H} \cup \{x_q+x_q=x_q\} \subseteq E_{{\rm card}({\cal T})}$.
Formally, this equivalence can be written as
\[
\forall x_1,\ldots,x_p \in \K ~\Bigl(D(x_1,\ldots,x_p)=0 \Longleftrightarrow
\exists x_{p+1},\ldots,x_{{\rm card}({\cal T})} \in \K
\]
\[
(x_1,\ldots,x_p,x_{p+1},\ldots,x_{{\rm card}({\cal T})}) {\rm ~solves~the~system~}
{\cal H} \cup \{x_q+x_q=x_q\} \Bigr)
\]
For each ring $\K$ extending $\Z$ and for each \mbox{$x_1,\ldots,x_p \in \K$} with
\mbox{$D(x_1,\ldots,x_p)=0$} there exists a unique tuple
\mbox{($x_{p+1},\ldots,x_{{\rm card}({\cal T})}) \in \K^{{\rm card}({\cal T})-p}$} such
that the tuple \mbox{$(x_1,\ldots,x_p,x_{p+1},\ldots,x_{{\rm card}({\cal T})})$} solves the system
\mbox{${\cal H} \cup \{x_q+x_q=x_q\}$}. Hence, for each ring $\K$ extending $\Z$ the equation
\mbox{$D(x_1,\ldots,x_p)=0$} has the same number of solutions as the system \mbox{${\cal H} \cup \{x_q+x_q=x_q\}$}.
\end{lemma}
\end{sloppypar}
\par
Putting $M=M/2$ we obtain new families ${\cal T}$ and ${\cal H}$.
There is a unique $q \in \{1,\ldots,{\rm card}({\cal T})\}$ such that
\[
\Bigl(q \in \{1,\ldots,p\}~ \wedge ~x_q=D(x_1,\ldots,x_p)\Bigr)~ \vee
\]
\[
\Bigl(q \in \{p+1,\ldots,{\rm card}({\cal T})\}~ \wedge ~\tau(q)=D(x_1,\ldots,x_p)\Bigr)
\]
The new system \mbox{${\cal H} \cup \{x_q+x_q=x_q\}$} is equivalent to \mbox{$D(x_1,\ldots,x_p)=0$}
and can be simply computed.
\vskip 0.2truecm
\par
It is unknown whether $\Z$ is existentially definable in $\Q$.
If it is, then a strong variant of the Bombieri-Lang conjecture is false, see \cite[p.~21,~Theorem~20]{Koenigsmann}.
\begin{theorem}\label{the3} (cf. \cite[p.~180,~Theorem~3.1]{Tyszka4}) Let \mbox{$f: \N \to \N$} be a computable
function. If $\Z$ is definable in $\Q$ by an existential formula, then there is
a positive integer $q$ and a system $S \subseteq E_q$ such that $S$ has infinitely many rational solutions
and they all belong to \mbox{${\Q}^q \setminus [-f(q),f(q)]^q$}.
\end{theorem}
\begin{proof}
If $\Z$ is definable in $\Q$ by an existential formula, then $\Z$ is definable in $\Q$ by a Diophantine formula. By Lemma~\ref{lem3},
\[
\forall t_1 \in \Q ~\Bigl(t_1 \in \Z \Longleftrightarrow \exists t_2,\ldots,t_p \in \Q ~~\Phi(t_1,t_2,\ldots,t_p)\Bigr)
\]
\begin{sloppypar}
\noindent
where $\Phi(t_1,t_2,\ldots,t_p)$ is a conjunction of formulae of the forms \mbox{$t_i=1$}, \mbox{$t_i+t_j=t_k$,}
\mbox{$t_i \cdot t_j=t_k$,} where $i,j,k \in \{1,\ldots,p\}$. The function \mbox{$\N \ni n \to f(n \cdot p)+1 \in \N$}
is computable. By Theorem~\ref{the1}, there is a positive integer $m$ and a system \mbox{$S \subseteq E_m$}
such that $S$ has infinitely many integer solutions and they all belong to \mbox{${\Z}^m \setminus [-f(m \cdot p),f(m \cdot p)]^m$}.
The following system
\end{sloppypar}
\[\left\{
\begin{array}{l}
{\rm all~equations~occurring~in~} S\\
{\rm all~equations~occurring~in~} \Phi(x_1,x_{1,2},\ldots,x_{1,p})\\
{\rm all~equations~occurring~in~} \Phi(x_2,x_{2,2},\ldots,x_{2,p})\\
\ldots\\
{\rm all~equations~occurring~in~} \Phi(x_{m-1},x_{m-1,2},\ldots,x_{m-1,p})\\
{\rm all~equations~occurring~in~} \Phi(x_m,x_{m,2},\ldots,x_{m,p})
\end{array}
\right.\]
with $m \cdot p$ variables has infinitely many rational solutions and they
all belong to ${\Q}^{m \cdot p} \setminus [-f(m \cdot p),f(m \cdot p)]^{m \cdot p}$.
\end{proof}
\newpage
For many Diophantine equations we know that the number of integer (rational) solutions is
finite, let us recall Siegel's theorem on integral points on curves and Faltings' theorem.
Faltings' theorem tell us that certain curves
have finitely many rational points, but no known proof gives any bound on the sizes of the
numerators and denominators of the coordinates of those points, see \cite[p.~722]{Gowers}.
In all such cases the Conjecture will allow us to compute such a bound.
\begin{theorem}\label{the4}
\begin{sloppypar}
Assuming the Conjecture, if a Diophantine equation \mbox{$D(x_1,\ldots,x_p)=0$}
has only finitely many integer solutions, then each such solution $(x_1,\ldots,x_p)$ satisfies
\end{sloppypar}
\[
|x_1|,\ldots,|x_p| \leq {\rm bound}(D)=2^{\textstyle 2^{\textstyle (2M+1)^{\textstyle (d_1+1) \cdot \ldots \cdot (d_p+1)}-1}}
\]
Here, $M$ stands for the maximum of the absolute values of the coefficients of $D(x_1,\ldots,x_p)$,
$d_i$ denote the degree of $D(x_1,\ldots,x_p)$ with respect to the variable~$x_i$.
\end{theorem}
\begin{proof}
It follows from $(\ast)$ and Lemma~\ref{lem3}.
\end{proof}
\begin{corollary}\label{cor1}
Assuming the Conjecture, for each polynomial $D(x_1,\ldots,x_p)$ with integer coefficients
\[
{\rm card} \left(\left\{(x_1,\ldots,x_p) \in {\Z}^p:~D(x_1,\ldots,x_p)=0\right\}\right) \in 
\]
\[
\left\{0,1,2,\ldots,\left(1+2 \cdot {\rm bound}(D)\right)^p \right\} \cup \left\{\omega\right\}
\]
\end{corollary}
\vskip 0.2truecm
\par
Unfortunately, it is undecidable whether a Diophantine equation has infinitely or finitely many solutions
in positive integers, see \cite{Davis1}. The same is true when we consider integer solutions or non-negative
integer solutions. Moreover, the set of Diophantine equations which have at most finitely many solutions in
non-negative integers is not recursively enumerable, see \mbox{\cite[p.~104,~Corollary~1]{Smorynski1}} and
\mbox{\cite[p.~240]{Smorynski2}}.
\vskip 0.2truecm
\par
For a polynomial $D(x,y)$ with integer coefficients, the following set
\[
\{x \in \N: \exists y \in \Z ~(D(x,y)=0 \vee D(-x,y)=0)\}~\cup
\]
\[
\{y \in \N: \exists x \in \Z ~(D(x,y)=0 \vee D(x,-y)=0)\}
\]
consists of non-negative integers. Let ${\bf Big}(D)$ denote its supremum in $\N$. Of course, \mbox{${\bf Big}: \Z[x,y] \to \N \cup \{\infty\}$}.
Let us consider the following three statements:
\vskip 0.2truecm
\noindent
~~~(1)~~~For each polynomial $D(x,y)$ with integer coefficients, it is decidable whether
or not the equation $D(x,y)=0$ has only finitely many solutions in \mbox{integers~$x,y$}.
\vskip 0.2truecm
\noindent
~~~(2)~~~The Conjecture.
\vskip 0.2truecm
\noindent
~~~(3)~~~The function {\bf Big} is not computable.
\vskip 0.2truecm
\noindent
Statement~(3) expresses the conjecture of J.~M.~Rojas, see \cite{Rojas2}, cf. \cite{Rojas1}. The negation of statement~(3)
implies statement~(1). By Theorem~\ref{the4}, \mbox{statements (1)--(3)} are jointly inconsistent.
\vskip 0.2truecm
\par
Assuming the Conjecture, if a Diophantine equation has only finitely many integer solutions,
then these solutions can be algorithmically found by applying Theorem~\ref{the4}.
Of course, only theoretically, because for interesting Diophantine
equations the bound $2^{\textstyle 2^{n-1}}$ is too high for the method of exhaustive search.
Usually, but not always. The equation $x_1^5-x_1=x_2^2-x_2$ has only finitely many rational solutions (\cite{Mignotte}),
and we know all integer solutions,
$(-1,0)$, $(-1,1)$, $(0,0)$, $(0,1)$, $(1,0)$, $(1,1)$, $(2,-5)$, $(2,6)$, $(3,-15)$, $(3,16)$, $(30,-4929)$,
$(30,4930)$, see~\cite{Bugeaud}. \mbox{Always} $x_2^2-x_2 \geq -\frac{1}{4}$, so $x_1>-2$.
The system
\begin{displaymath}
\left\{
\begin{array}{rcl}
x_1 \cdot x_1 &=& x_3 \\
x_3 \cdot x_3 &=& x_4 \\
x_1 \cdot x_4 &=& x_5 \\
x_1+x_6 &=& x_5 \\
x_2 \cdot x_2 &=& x_7 \\
x_2+x_6 &=& x_7
\end{array}
\right.
\end{displaymath}
\begin{sloppypar}
\noindent
is equivalent to \mbox{$x_1^5-x_1=x_2^2-x_2$}. By the Conjecture, \mbox{$|x_1^5|=|x_5| \leq 2^{\textstyle 2^{7-1}}=2^{64}$}.
Therefore, \mbox{$-2<x_1 \leq 2^{\textstyle \frac{64}{5}}<7132$}, so the equivalent equation \mbox{$4x_1^5-4x_1+1=(2x_2-1)^2$}
can be solved by a computer.
\end{sloppypar}
\vskip 0.2truecm
\par
The algorithm presented in the proof of Lemma~\ref{lem3} is not useful for practical computations, because it introduces
a large number of auxiliary variables. Therefore, for the equation \mbox{$x_1^5-x_1=x_2^2-x_2$} we have chosen the equivalent
system which has only $7$ variables. In \cite[pp.~92--93]{Cipu}, M.~Cipu studies the system
\begin{displaymath}
\left\{
\begin{array}{rcl}
x^2 - 3z^2 &=& 1 \\
y^2 - 783z^2 &=& 1
\end{array}
\right.
\end{displaymath}
for which he constructs various equivalent systems which contain only equations of the forms \mbox{$x_i=1$}, \mbox{$x_i+x_j=x_k$},
\mbox{$x_i \cdot x_j=x_k$}.
\vskip 0.2truecm
\par
Assuming the Conjecture, also the heights of rational solutions can be computably bounded from above,
as we will show in Theorem~\ref{the5}.
\begin{lemma}\label{lem4}~(\cite[p.~14,~the proof~of~Theorem~1.11]{Narkiewicz})~ The integers $A$ and $B>0$ are relatively prime if
and only if there exist integers $X$ and $Y$ such that $A \cdot X + B \cdot Y=1$ and $|X| \leq B$.
\end{lemma}
\begin{theorem}\label{the5}
\begin{sloppypar}
Assuming the Conjecture, if a Diophantine equation \mbox{$D(x_1,\ldots,x_p)=0$}
has only finitely many rational solutions, then their heights are bounded from above by a computable function of $D$.
\end{sloppypar}
\end{theorem}
\par
\begin{proof}
By applying Lemma~\ref{lem3}, we can write the equation as an equivalent
system $S \subseteq E_n$, where $n$ and $S$ can be computed.
We substitute \mbox{$x_m=\frac{\textstyle y_m}{\textstyle z_m}$} for \mbox{$m \in \{1,\ldots,n\}$}.
Each equation $x_i=1 \in S$ we replace by the equation $y_i=z_i$.
Each equation $x_i+x_j=x_k \in S$ we replace by the equation
$y_i \cdot z_j \cdot z_k+y_j \cdot z_i \cdot z_k = y_k \cdot z_i \cdot z_j$.
Each equation $x_i \cdot x_j=x_k \in S$ we replace by the equation
$(y_i \cdot z_j \cdot z_k) \cdot (y_j \cdot z_i \cdot z_k) = y_k \cdot z_i \cdot z_j$.
Next, we incorporate to $S$ all equations
\begin{displaymath}
\begin{array}{rcl}
1+s_m^2+t_m^2+u_m^2+v_m^2 &=& z_m \\
p_m \cdot y_m + q_m \cdot z_m &=& 1 \\
p_m^2+a_m^2+b_m^2+c_m^2+d_m^2 &=& z_m^2
\end{array}
\end{displaymath}
with $m \in \{1,\ldots,n\}$. By Lagrange's four-square theorem and Lemma~\ref{lem4},
the enlarged system has at most finitely many integer solutions and is equivalent to the original one. Next,
we construct a single Diophantine equation equivalent to the enlarged \mbox{system $S$}.
For this equation we apply Theorem~\ref{the4}
\end{proof}
\begin{theorem}\label{the6}
\begin{sloppypar}
Assuming the Conjecture, if a Diophantine equation \mbox{$D(x_1,\ldots,x_p)=0$}
has only finitely many solutions in non-negative integers, then the conjectural bound for these solutions
can be computed by applying Theorem~\ref{the4} to the equation
\end{sloppypar}
\[
\widehat{D}(x_1,a_1,b_1,c_1,d_1,\ldots,x_p,a_p,b_p,c_p,d_p)=
\]
\[
D(x_1,\ldots,x_p)^2+(x_1-a_1^2-b_1^2-c_1^2-d_1^2)^2+\ldots+(x_p-a_p^2-b_p^2-c_p^2-d_p^2)^2=0
\]
\end{theorem}
\begin{proof}
By Lagrange's four-square theorem
\[
\{(x_1,\ldots,x_p) \in {\Z}^p:~\exists a_1,b_1,c_1,d_1,\ldots,a_p,b_p,c_p,d_p \in \Z
\]
\[
\widehat{D}(x_1,a_1,b_1,c_1,d_1,\ldots,x_p,a_p,b_p,c_p,d_p)=0\}=
\]
\[
\{(x_1,\ldots,x_p) \in {\N}^p:~D(x_1,\ldots,x_p)=0\}
\]
Since the equation \mbox{$D(x_1,\ldots,x_p)=0$} has only finitely many solutions in non-negative integers,
the equation
\[
\widehat{D}(x_1,a_1,b_1,c_1,d_1,\ldots,x_p,a_p,b_p,c_p,d_p)=0
\]
has only finitely many solutions in integers $x_1,a_1,b_1,c_1,d_1,\ldots,x_p,a_p,b_p,c_p,d_p$.
\end{proof}
\begin{sloppypar}
M.~Davis, Yu.~Matiyasevich and J.~Robinson conjecture that there is no \mbox{algorithm} for listing the Diophantine
equations with infinitely many solutions, see \mbox{\cite[p.~372]{Davis2}}.
\end{sloppypar}
\begin{theorem}\label{the7}
The Conjecture implies that the set of Diophantine equations which have infinitely many solutions
in integers (non-negative integers) is recursively enumerable.
\end{theorem}
\begin{proof}
The following algorithm works for all polynomials \mbox{$D(x_1,\ldots,x_p) \in {\Z}[x_1,\ldots,x_p]$}.
\par
\noindent
\centerline{$\alpha:={\rm bound}(D)+1$}
\par
\noindent
\centerline{WHILE}
\par
\noindent
\centerline{$D(y_1,\ldots,y_p) \neq 0$ for all integers $y_1,\ldots,y_p$ with ${\rm max}(|y_1|,\ldots,|y_p|)=\alpha$}
\par
\noindent
\centerline{DO}
\par
\noindent
\centerline{$\alpha:=\alpha+1$}
\begin{sloppypar}
\noindent
Assuming the Conjecture and applying Theorem~\ref{the4}, we conclude that the algorithm terminates
if and only if the equation \mbox{$D(x_1,\ldots,x_p)=0$} has infinitely many solutions in integers.
For solutions in non-negative integers, we consider the following algorithm:
\end{sloppypar}
\par
\noindent
\centerline{$\theta:={\rm bound}(\widehat{D})+1$}
\par
\noindent
\centerline{WHILE}
\par
\noindent
\centerline{$D(y_1,\ldots,y_p) \neq 0$ for all non-negative integers $y_1,\ldots,y_p$ with ${\rm max}(y_1,\ldots,y_p)=\theta$}
\par
\noindent
\centerline{DO}
\par
\noindent
\centerline{$\theta:=\theta+1$}
\par
\noindent
Assuming the Conjecture and applying Theorem~\ref{the6}, we conclude that the algorithm terminates
if and only if the equation \mbox{$D(x_1,\ldots,x_p)=0$} has infinitely many solutions in non-negative integers.
\end{proof}
\begin{theorem}\label{the8}
If Matiyasevich's conjecture is true, then for every computable function
\mbox{$f:\N \setminus \{0\} \to \N$} there is a positive integer $m(f)$ such that
for each integer \mbox{$n \geq m(f)$} there exists a system \mbox{$S \subseteq E_n$} which has
only finitely many solutions in integers \mbox{$x_1,\ldots,x_n$} and
each integer tuple \mbox{$(x_1,\ldots,x_n)$} that solves~$S$ satisfies \mbox{$x_1=f(n)+1$}.
\end{theorem}
\begin{proof}
By Matiyasevich's conjecture, the function $\N \setminus \{0\} \ni n \to f(n)+1 \in \N$ has a finite-fold Diophantine
representation. It means that there is a polynomial $W(x_1,x_2,x_3,\ldots,x_r)$ with integer coefficients such that
for each non-negative integers $x_1$, $x_2$,
\begin{equation}
\tag*{\tt (E1)}
\Bigl(x_2 \geq 1 \wedge x_1=f(x_2)+1\Bigr) \Longleftrightarrow \exists x_3, \ldots, x_r \in \N ~~W(x_1,x_2,x_3,\ldots,x_r)=0
\end{equation}
and
\[
{\rm only~finitely~many~tuples~}(x_3,\ldots,x_r) \in {\N}^{r-2} {\rm ~satisfy~} W(x_1,x_2,x_3,\ldots,x_r)=0~~~~{\tt (A)}.
\]
By the equivalence~{\rm (E1)} and Lagrange's four-square theorem, for each \mbox{integers $x_1$, $x_2$},
the conjunction \mbox{$(x_2 \geq 1) \wedge (x_1=f(x_2)+1)$} holds true if and only if there exist integers
\[
a,b,c,d,\alpha,\beta,\gamma,\delta,x_3,x_{3,1},x_{3,2},x_{3,3},x_{3,4},\ldots,x_r,x_{r,1},x_{r,2},x_{r,3},x_{r,4}
\]
such that
\[
W^2(x_1,x_2,x_3,\ldots,x_r)+\bigl(x_1-a^2-b^2-c^2-d^2\bigr)^2+\bigl(x_2-\alpha^2-\beta^2-\gamma^2-\delta^2\bigr)^2+
\]
\[
\bigl(x_3-x^2_{3,1}-x^2_{3,2}-x^2_{3,3}-x^2_{3,4}\bigr)^2+\ldots+\bigl(x_r-x^2_{r,1}-x^2_{r,2}-x^2_{r,3}-x^2_{r,4}\bigr)^2=0
\]
The sentence~{\tt (A)} guarantees that for each integers $x_1$, $x_2$, only finitely many integer tuples
\[
\bigl(a,b,c,d,\alpha,\beta,\gamma,\delta,x_3,x_{3,1},x_{3,2},x_{3,3},x_{3,4},\ldots,x_r,x_{r,1},x_{r,2},x_{r,3},x_{r,4}\bigr)
\]
satisfy the last equality. By Lemma~\ref{lem3}, there is an integer \mbox{$s \geq 3$} such that for each integers $x_1$, $x_2$,
\begin{equation}
\tag*{\tt (E2)}
\Bigl(x_2 \geq 1 \wedge x_1=f(x_2)+1\Bigr) \Longleftrightarrow \exists x_3,\ldots,x_s \in \Z ~~\Psi(x_1,x_2,x_3,\ldots,x_s)
\end{equation}
where the formula $\Psi(x_1,x_2,x_3,\ldots,x_s)$ is algorithmically determined as a conjunction of formulae of the forms
\mbox{$x_i=1$}, \mbox{$x_i+x_j=x_k$}, \mbox{$x_i \cdot x_j=x_k$} \mbox{($i,j,k \in \{1,\ldots,s\}$)} and
for each integers \mbox{$x_1,x_2$} at most finitely many integer tuples \mbox{$(x_3,\ldots,x_s)$} satisfy
\mbox{$\Psi(x_1,x_2,x_3,\ldots,x_s)$}.
Let $m(f)=4+2s$, and let $[\cdot]$ denote the integer part function. For each integer $n \geq m(f)$,
\[
n-\left[\frac{n}{2}\right]-2-s \geq m(f)-\left[\frac{m(f)}{2}\right]-2-s \geq m(f)-\frac{m(f)}{2}-2-s=0
\]
Let $S$ denote the following system
\[\left\{
\begin{array}{rcl}
{\rm all~equations~occurring~in~}\Psi(x_1,x_2,x_3,\ldots,x_s) \\
n-\left[\frac{n}{2}\right]-2-s {\rm ~equations~of~the~form~} z_i=1 \\
t_1 &=& 1 \\
t_1+t_1 &=& t_2 \\
t_2+t_1 &=& t_3 \\
&\ldots& \\
t_{\left[\frac{n}{2}\right]-1}+t_1 &=& t_{\left[\frac{n}{2}\right]} \\
t_{\left[\frac{n}{2}\right]}+t_{\left[\frac{n}{2}\right]} &=& w \\
w+y &=& x_2 \\
y+y &=& y {\rm ~(if~}n{\rm ~is~even)} \\
y &=& 1 {\rm ~(if~}n{\rm ~is~odd)}
\end{array}
\right.\]
with $n$ variables. The system $S$ has only finitely many integer solutions,
By the equivalence~{\tt (E2)}, $S$ is consistent over $\Z$.
If an integer $n$-tuple $(x_1,x_2,x_3,\ldots,x_s,\ldots,w,y)$ solves~$S$,
then by the equivalence~{\tt (E2)},
\[
x_1=f(x_2)+1=f(w+y)+1=f\left(2 \cdot \left[\frac{n}{2}\right]+y\right)+1=f(n)+1
\]
\end{proof}
\begin{sloppypar}
\begin{corollary}\label{cor2}
The Conjecture formulated for an arbitrary computable bound
\mbox{$f: \N \setminus \{0\} \to \N$} instead of the bound \mbox{$\N \setminus \{0\} \ni n \to 2^{\textstyle 2^{n-1}} \in \N$}
remains in contradiction to Matiyasevich's conjecture on finite-fold Diophantine representations.
\end{corollary}
\end{sloppypar}
\par
Logicians believe in Matiyasevich's conjecture, but some heuristic argument suggests the opposite possibility.
Below is the excerpt from page 135 of the \mbox{book \cite{Smart}:}
\vskip 0.2truecm
\par
\noindent
{\em Folklore. If a Diophantine equation has only finitely many solutions then those solutions
are small in `height' when compared to the parameters of the equation.
This folklore is, however, only widely believed because of the large amount of experimental
evidence which now exists to support it.}
\vskip 0.2truecm
\par
Below is the excerpt from page 12 of the article \cite{Stoll}:
\par
\vskip 0.2truecm
\noindent
{\em Note that if a Diophantine equation is solvable, then we can prove it, since we will
eventually find a solution by searching through the countably many possibilities
(but we do not know beforehand how far we have to search). So the really hard
problem is to prove that there are no solutions when this is the case. A similar
problem arises when there are finitely many solutions and we want to find them
all. In this situation one expects the solutions to be fairly small. So usually it
is not so hard to find all solutions; what is difficult is to show that there are no
others.}
\vskip 0.2truecm
\par
That is, mathematicians are intuitively persuaded that solutions are small when there are finitely
many of them. It seems that there is a reason which is common to all the equations. Such a
reason might be the Conjecture whose consequences we have already presented.

\par
\noindent
Apoloniusz Tyszka\\
Technical Faculty\\
Hugo Ko\l{}\l{}\k{a}taj University\\
Balicka 116B, 30-149 Krak\'ow, Poland\\
E-mail address: \url{rttyszka@cyf-kr.edu.pl}
\end{document}